\documentclass[11pt]{amsart}

\usepackage{graphicx, graphics}

\usepackage{hyperref}

\emergencystretch=\maxdimen
\usepackage{amsmath}
\usepackage{amsthm}
\usepackage{amsfonts}
\usepackage{color}
\usepackage[english]{babel}
\usepackage[margin=1.5in]{geometry}

\usepackage{enumerate}
\usepackage{subcaption}

        \definecolor{brown}{rgb}{1,0,1}

\usepackage[latin9]{inputenc}
\usepackage{amsmath}
\usepackage{amsfonts}
\usepackage{amssymb}
\usepackage{amsthm}

\numberwithin{equation}{section}

\newtheorem{theo}{Theorem}[section]

\theoremstyle{definition}

\theoremstyle{remark}
\newtheorem{rema}[theo]{Remark}

\newcommand{\N}{{\mathbb N}}
\newcommand{\nwc}{\newcommand}
\nwc{\eps}{\epsilon}
\nwc{\vareps}{\varepsilon}
\nwc{\Oph}{\operatorname{Op}_\hbar}
\nwc{\ra}{\rangle}
\nwc{\la}{\lambda}

\nwc{\mf}{\mathbf} 
\nwc{\blds}{\boldsymbol} 
\nwc{\ml}{\mathcal} 

\nwc{\defeq}{\stackrel{\rm{def}}{=}}

\nwc{\cE}{\ml{E}}
\nwc{\cN}{\ml{N}}
\nwc{\cO}{\ml{O}}
\nwc{\cP}{\ml{P}}
\nwc{\cU}{\ml{U}}
\nwc{\cV}{\ml{V}}
\nwc{\cW}{\ml{W}}
\nwc{\tU}{\widetilde{U}}
\nwc{\IN}{\mathbb{N}}
\nwc{\IR}{\mathbb{R}}
\nwc{\IZ}{\mathbb{Z}}
\nwc{\IC}{\mathbb{C}}
\nwc{\IT}{\mathbb{T}}
\nwc{\tP}{\widetilde{P}}
\nwc{\tPi}{\widetilde{\Pi}}
\nwc{\tV}{\widetilde{V}}
\nwc{\supp}{\operatorname{supp}}
\nwc{\rest}{\restriction}

\renewcommand{\phi}{\varphi}

\newcommand{\ep}{\varepsilon}

\newtheorem{cor}[theo]{{\sc Corollary}}
\newtheorem{lem}[theo]{{\sc Lemma}}

\title [Inner radius of nodal domains of quantum ergodic eigenfunctions] {Inner radius of nodal domains of quantum ergodic eigenfunctions }

\author{Hamid Hezari }
\address{Department of Mathematics, UC Irvine, Irvine, CA 92617, USA} \email{hezari@math.uci.edu}

\begin{document}


\maketitle

\begin{abstract}
In this short note we show that the lower bounds of Mangoubi \cite{Ma2} on the inner radius of nodal domains can be improved for quantum ergodic sequences of eigenfunctions, according to a certain power of the radius of shrinking balls on which the eigenfunctions equidistribute. We prove such improvements using a quick application of our recent results \cite{He16}, which give modified growth estimates for eigenfunctions that equidistribute on small balls.  Since by \cite{Han15, HeRi} small scale QE holds for negatively curved manifolds on logarithmically shrinking balls, we get logarithmic improvements on the inner radius of eigenfunctions on such manifolds. We also get improvements for manifolds with ergodic geodesic flows. In addition using the small scale equidistribution results of \cite{LeRu},  one gets polynomial betterments of \cite{Ma2} for toral eigenfunctions in dimensions $n \geq 3$. The results work only for a full density subsequence of eigenfunctions.

\end{abstract}

\section{Introduction} Let $(X,g)$ be a smooth compact connected boundaryless Riemannian manifold of dimension $n$.  Suppose $\Delta_g$ is the positive Laplace-Beltrami operator on $(X, g)$ and $\psi_\lambda$ is a real-valued eigenfunction of $\Delta_g$ with eigenvalue $\lambda >0$. Let $\Omega_\lambda$ be a nodal domain of $\psi_\lambda$ and $\text{in}(\Omega_\lambda)$ be its inner radius. Mangoubi \cite{Ma1, Ma2} has shown that \footnote{In particular in dimension two, one has the optimal lower bound $a_1 \lambda^{-1/2}$.} 
\begin{equation} \label{Bounds}{a_1}{ \la^{-\frac{1}{2}-\frac{(n-1)(n-2)}{4n} }} \leq \text{in}( \Omega_\la ) \leq {a_2} {\la^{-\frac{1}{2}}}, \end{equation}
where $a_1$ and $a_2$ depend only on $(X,g)$. In this note we show that
\begin{theo} \label{LowerBoundIR} There exists $r_0(g) >0$ such that if $ \lambda^{-1/2} < r_0(g)$, and if for some $r \in [\lambda^{-1/2}, r_0(g)]$ and for all geodesic balls $\{B_{r}(x)\}_{x \in X}$ we have
\begin{equation} \label{SmallScaleAssumption}
K_1 r^n\leq \int _{B_{r}(x)} | \psi_{\lambda}|^2 dv_g \leq K_2 r^n,
\end{equation} for some positive  constants $K_1$ and $K_2$ independent of $x$, 
then for $n \geq 3$
\begin{equation}\label{RefinedIR} {a_1} r^{-\frac{(n-1)(n-2)}{2n}}{ \la^{-\frac{1}{2}-\frac{(n-1)(n-2)}{4n} }} \leq \text{in}( \Omega_\la ) \end{equation}
\end{theo}

A result of \cite{HeRi}\footnote{In \cite{Han15}, this is proved for $\kappa \in (0, \frac{1}{2n})$.}, shows that on negatively curved manifolds, (\ref{SmallScaleAssumption}) holds for a full density subsequence  with $r= (\log \lambda)^{-\kappa}$ for any $\kappa \in (0, \frac{1}{2n})$. Hence the following unconditional result on such manifolds is quickly obtained.

\begin{cor}\label{LowerBound-NC}
Let $(X, g)$ be a boundaryless compact connected smooth Riemannian manifold of dimension $n \geq 3$, with negative sectional curvatures.  Let $\{ \psi_{\lambda_j} \}_{j \in \IN}$ be any ONB of $L^2(X)$ consisting of real-valued eigenfunctions of $\Delta_g$ with eigenvalues $\{\lambda_j\}_{j \in \IN}$. Let $\ep >0$ be arbitrary. Then there exists $S_\ep \subset \IN$ of \textit{full density} \footnote{It means that $ \lim_{N \to \infty} \frac{1}{N} \text{card} \big ( S \cap [1, N] \big )  =1$.} such that for $j \in S_\ep$:
$$ {a_1} (\log \la_j)^{\frac{(n-1)(n-2)}{4n^2} -\ep}{ \la_j^{-\frac{1}{2}-\frac{(n-1)(n-2)}{4n} }} \leq \text{in}( \Omega_\la ),$$ for some $a_1$ that only depends on $(X, g)$ and $\ep$. 
\end{cor}

One can also get the following improvements for quantum ergodic sequences of eigenfunctions. In fact it is enough to assume equidistribution on the configuration space $X$ .  

\begin{cor} \label{LowerBoundQE} Let $(X, g)$ be a boundaryless compact connected smooth Riemannian manifold of dimension $n \geq 3 $.  Let $\{ \psi_{\lambda_j}\}_{j \in S}$ be a  sequence of  real-valued eigenfunctions of $\Delta_g$ with eigenvalues $\{\lambda_j\}_{j \in S}$ such that for all $r \in (0, R_0)$, for some fixed $R_0 >0$, and all $x \in X$
\begin{equation} \label{QEonX}
\int_{B_r(x)} |\psi_{\lambda_j}|^2 \to \frac{\text{Vol}_g(B_r(x))}{\text{Vol}_g(X)}, \qquad  \lambda_j \xrightarrow{j \in S} \infty.
\end{equation}
Then along this sequence $$ \lim_{ j \to \infty} \la_j^{\frac{1}{2}+\frac{(n-1)(n-2)}{4n}}  \text{in}(\Omega_{\la_j}) = \infty.$$  
\end{cor}
In particular the above corollary holds for manifolds with ergodic geodesic flows by the quantum ergodicity theorem of \cite{Sh, CdV, Ze87}. This means that given any ONB of eigenfunctions on such a manifold one can find a full density subsequence where (\ref{QEonX}), hence Corollary \ref{LowerBoundQE} holds. 

We must also mention the work of Lester-Rudnick \cite{LeRu} where they proved that for a full density subsequence of toral eigenfunctions one has equidistribution  at the shrinking rate  $r= \lambda^{- \frac{1}{2n-2} +\epsilon}$. Of course, one can also use this and Theorem \ref{LowerBoundIR} to get improved lower bounds for such toral eigenfunctions.

\begin{cor} \label{Toral} For any ONB $\{\psi_{\lambda_j} \}_ {\j \in \N}$ of real-valued eigenfunctions of $\Delta$ on the flat torus $\mathbb T^{n \geq 3}$, and any $\ep >0$, there exists a full density subset $S_\ep \subset \N$ such that for $j \in S_\ep$
$${a_1} { \la_j^{-\frac{1}{2}-\frac{(n-2)^2}{4n} -\ep }} \leq \text{in}( \Omega_{\la_j} ),$$ where $a_1$ is a positive constant that depends only on $n$ and $\ep$. 
\end{cor}

\subsection{Proofs of Theorem \ref{LowerBoundIR}} The main idea is to use the modified growth estimates of our recent preprint \cite{He16}. We recall that Donnelly-Fefferman \cite{DF1} showed that an eigenfunction $\psi_\lambda$ of $\Delta_g$ with eigenvalue $\lambda$ satisfies
$$ N(B_s(x)):= \log \left ( \frac{\sup_{B_{2s}(x)} |\psi_{\lambda}|^2} { \sup_{B_{s}(x)} |\psi_{\lambda}|^2}  \right ) \leq c \sqrt{\lambda},$$ for all $s < s_0$ where $s_0$ and $c$ depend only on $(X, g)$.  In \cite{He16}, we have shown that  
under the assumption
$$K_1 r^n\leq \int _{B_{r}(x)} | \psi_{j}|^2 \leq K_2 r^n, $$  we have
\begin{equation}\label{Doubling} N(B_s(x)) \leq c \, r \sqrt{\lambda}, \qquad \text{for all $s < 10 r$}, \end{equation} where $c$ is positive and is uniform in $x$, $r$, $s$, and $\lambda$, but depends on $K_1$, $K_2$, and $(X, g)$. We apply these growth estimates to the proof of \cite{Ma2}. 

We emphasize that  in what follows $n \geq 3$. In \cite{Ma2}, it is first proved that there exists $\ep_0 \in (0, 1)$, sufficiently small and only dependent on $(X, g)$, such that for all embedded geodesic balls $B_R(p)$ with $ \{ \psi_\la =0 \} \cap B_{R/2}(p) \neq \emptyset$ we have 
$$  \frac{ \text{Vol} \big ( \{ \psi_\la \geq 0 \} \cap B_R(p)  \big )}{ \text{Vol} (B_R(p))} \geq \frac{a_3}{N(\la)^{n-1}}, $$
where $$N(\lambda)= \sup_{x \in X, s \leq \ep_0 \lambda^{-\frac 12}} N(B_s(x)),$$ and $a_3 >0$ depends only on $(X, g)$. Then it is shown that
$$ \text{in}(\Omega_\lambda) \geq a_4 \la^{- \frac12}  \inf _{B \in \mathcal B_\la} \left ( \frac{ \text{Vol} \big ( \{ \psi_\la \geq 0 \} \cap B_R(p)  \big )}{ \text{Vol} (B_R(p))} \right )^{\frac{n-2}{2n}}, $$
where the infimum in taken over the set $\mathcal B_\la$ of  balls $B_R(p)$ such that $ \{ \psi_\la =0 \} \cap B_{R/2}(p) \neq \emptyset$, and $a_4 >0$ depends only on $(X, g)$. Combining the last two inequalities, one obtains
$$\text{in}(\Omega_\lambda) \geq a_5 \la^{- \frac12}  N(\lambda)^{-\frac{(n-1)(n-2)}{2n}}.  $$
Since $s\leq \ep_0 \la^{-1/2} \leq r <10r$, we can use our improved doubling estimates (\ref{Doubling}) to get $N(\lambda) \leq c r \sqrt{\la}$, which implies the theorem easily.

\subsection{Proof of of Corollary \ref{LowerBoundQE}}

This  follows quickly from the following lemma together with Theorem \ref{LowerBoundIR}. 
\begin{lem} \label{QElemma} Let $\{ \psi_j \}_{j \in S}$ be a  sequence of eigenfunctions of $\Delta_g$ with eigenvalues $\{\lambda_j\}_{j \in S}$ such that for some $R_0 >0$, all $r \in (0, R_0)$, and all $x \in X$
\begin{equation} \label{QEonX2}
\int_{B_r(x)} |\psi_{j}|^2 \to \frac{\text{Vol}_g(B_r(x))}{\text{Vol}_g(X)}, \qquad  \lambda_j \xrightarrow{j \in S} \infty.
\end{equation}
Then there exists $r_0(g)$ such that for each $r \in (0, r_0(g))$ there exists $\Lambda_r$ such that for $ \lambda_j \geq \Lambda_r$ we have
$$K_1 r^n\leq \int _{B_{r}(x)} | \psi_{j}|^2 \leq K_2 r^n, $$ uniformly for all $x \in X$. Here, $K_1$ and $K_2$ are independent of $r$, $j$, and $x$. 
 \end{lem}

\begin{proof} This lemma is obvious when $x$ is fixed but to show that it holds uniformly in $x$ one can use a covering argument as presented  in  \cite{He16}. 
\end{proof}

\begin{rema} It is clear from Theorem \ref{LowerBoundIR} that if (\ref{SmallScaleAssumption}) holds for $r=\la^{-\frac 12 + \ep}$ for arbitrary $\ep>0$, then we have  $${a_1(\ep)}{ \la^{-\frac{1}{2}-\ep }} \leq \text{in}( \Omega_\la ) \leq {a_2} {\la^{-\frac{1}{2}}}, $$ for any $\ep>0$. It is a natural, but seemingly very difficult to prove, conjecture that $r=\la^{-\frac 12 + \ep}$ is the optimal rate for the eigenfunctions of negatively curved manifolds. A result of \cite{LeRu} shows that this optimal rate of shrinking is satisfied for a full density subsequence of any ONB of eigenfunctions of the flat $2$-torus  $\mathbb T^2$. 

\end{rema}

\begin{rema} We have used both local and global harmonic analysis of eigenfunctions. The local analysis is contained in the work of Mangoubi, where lower bounds for the inner radius in terms of the doubling index are given. The global analysis is contained in the small scale QE results of \cite{Han15, HeRi} and also the QE results of \cite{Sh, CdV, Ze87}, where the global behavior of the geodesic flow and the wave operator are considered to obtain equidistribution on small balls.  Hence, any non-optimal results that give improvements on the results of \cite{Ma2} and is only based on the local analysis of eigenfunctions can be improved using our hybrid method. 

\end{rema}

\begin{rema} The upper bound in (\ref{Bounds}) is an immediate consequence of a  result of Br\"uning \cite{B}, which says that there exists $a_2 >0$ dependent only on $(X, g)$ such that every geodesic ball of radius $a_2 \lambda^{-1/2}$ contains a zero of $\psi_\lambda$. See \cite{Ze08} for a simple proof using the domain monotonicity  property of Dirichlet eigenvalues.  

\end{rema}

\begin{rema} The paper \cite{GeMa} gives a refinement of the result of Mangoubi using Brownian motions techniques. To be precise, the authors show that the inscribed ball that gives the lower bound of \cite{Ma2} can be centered at the maximum of the eigenfunctions $\psi_\la$ on its nodal domain $\Omega_\la$.   
\end{rema}

\section*{Acknowledgments} We are grateful to Steve Zelditch for encouraging us to write this note.

\end{document}